\documentclass[document,11pt]{article}
\usepackage{amsmath,multicol}
\usepackage{epsf, subfigure, verbatim, epsfig}
\usepackage{fancyhdr}
\usepackage{geometry}
\usepackage{calc}
\usepackage{ifthen}
\usepackage{layout}
\usepackage{fancybox}
\usepackage{eurosym}
\usepackage{tabularx}
\usepackage{xspace}
\usepackage{dsfont,mathrsfs}
\usepackage{amssymb}
\usepackage{amsthm}
\usepackage[utf8]{inputenc}
\usepackage[english]{babel}
\usepackage[backend=biber,style = alphabetic]{biblatex}
\usepackage{mathtools}
\usepackage{amsthm}
\usepackage{amssymb}
\usepackage{fixltx2e}
\usepackage[T1]{fontenc}
\usepackage{newpxtext}
\usepackage[varg,bigdelims]{newpxmath}
\usepackage[cal=euler,scr=rsfso]{mathalfa}
\usepackage{bm}
\usepackage{inputenc}
\usepackage{microtype}
\usepackage[usenames,dvipsnames]{xcolor}
\usepackage{paralist}
\usepackage{booktabs}
\usepackage{tikz}
\usepackage{todonotes}
\usepackage[bookmarks=true,colorlinks=true, linkcolor=MidnightBlue, citecolor=cyan]{hyperref}
\usepackage{tabularx}
\usepackage{multirow}
\usepackage{makecell}

\usepackage{thmtools}
\usepackage{thm-restate}

\usepackage{cleveref}

\usetikzlibrary{
	cd,
	decorations.markings,
	positioning,
	arrows.meta,
	shapes,
	calc,
	fit,
	quotes}
\hypersetup{final}

\tikzset{
  tick/.style={postaction={
    decorate,
    decoration={markings, mark=at position 0.5 with {\draw[-] (0,.4ex) -- (0,-.4ex);}}}
}}

\newtheorem{lemma}{Lemma}[section]
\theoremstyle{definition}
\newtheorem{definition}{Definition}[section]

\newcommand{\mbb}[1]{\mathbb{#1}}

\DeclareMathOperator{\Id}{Id}
\DeclareMathOperator{\Prof}{Prof}

\DeclareMathOperator{\Set}{Set}
\DeclareMathOperator{\RK}{RK}
\DeclareMathOperator{\CS}{CS}
\DeclareMathOperator{\DS}{DS}

\fancyhfoffset[R]{0in}

\begin{document}

\large
\begin{center}
{\bf Compositionality of the Runge-Kutta Method}
\end{center}

\normalsize

\begin{abstract}

In Spivak \cite{spivak}, dynamical systems are described in terms of their inputs and outputs in a pictorial way using an operad of wiring diagrams.
Each dynamical system is a box with certain inputs and outputs, and multiple dynamical systems are linked together using wiring diagrams, which describe how the outputs of one dynamical system to the inputs of another.
By describing dynamical systems in this way, we can decompose a large dynamical system as a collection of smaller, simpler dynamical systems linked together.
Of course, this decomposition is only useful if we can work with these smaller, simpler dynamical systems instead of the larger one.
In his paper, Spivak shows that we can perform Euler's method on these smaller systems and still get the same results as working on the larger one.
In this paper, we extend his results to prove that we can do something similar with the Runge-Kutta method.
However, we need to modify the framework used in Spivak's paper to account for the fact that the Runge-Kutta method requires multiple steps, unlike Euler's method.
To better describe these systems, we define wiring diagrams as objects of a double-category and dynamical systems in terms of double functors, giving a categorical description of this approximation method.

\end{abstract}

\section{Introduction} \label{introduction}

Dynamical systems are everywhere, from the cells in our bodies to the computers which we work on.
Each of these machines take inputs, modify an internal state, and then give an output.
For a computer, the inputs could be user inputs on a keyboard, which then modify the binary state held in the hard drive and memory, and the output could be the display on the monitor.
In addition, by using the outputs of and providing the inputs for other dynamical systems, dynamical systems can become interconnected through a network which as a whole represents one larger dynamical system, such as the internet, which is a dynamical system composed of many interconnected computers.
When working with such a network, often it would be much easier to work with the individual dynamical systems and then combine the results instead of composing all the dynamical systems first and then working with that large system.

In Spivak's work \cite{spivak}, he provides a framework for describing dynamical systems and their interconnections through wiring diagrams.
A dynamical system with input set $A$ and output set $B$ is described as a tuple $(S,f^{rdt},f^{upd})$, with $S$ the state set, $f^{rdt}: S \to B$ a readout function turning the state to an output, and $f^{upd}: A \times S \to S$ an update function changing the state depending on the input.
For a discrete dynamical system, we let our state set be any set, and we define $f^{upd}$ by taking an input and state and give back the next state which should result.
For a continuous dynamical system, we let our state set be a real vector space and define $f^{upd}$ so that it takes the input and state and returns the infinitesimal change (like a derivative) which should result.
Often, we will describe a dynamical system as a diagram
\[
\begin{tikzcd}
A \times S \arrow[r] & S \arrow[r] & B
\end{tikzcd}
\]
with the arrows representing the update and readout functions.

In his paper, Spivak proves that Euler's method works in a natural way when converting continuous dynamical systems to discrete dynamical systems so that we can work with the individual dynamical systems before composing instead of having to work with the much larger combined dynamical system.
We would like to extend that result to the Runge-Kutta method, but unfortunately, the discrete dynamical system format above does not capture the four-step 
approach that Runge-Kutta uses.
In every iteration of Runge-Kutta, there are four computations performed, each depending on all the previous results of the iteration, but this iterative cycle of four steps is not captured in the original definition of a discrete dynamical system, as it only goes step by step and does not remember any states besides the previous one.
In addition, Spivak describes his continuous dynamical systems using the manifold structure of a real vector space, but for Runge-Kutta, it is the vector space properties which are more important.
As a result, we need to describe dynamical systems differently to more accurately capture this information.
In particular, we introduce a second type of morphism which describes a dynamical system working in line with another dynamical system.
An iterative four-step dynamical system can be thought of as a dynamical system which lines up with the cyclic four-step dynamical system $C_4$ shown below.
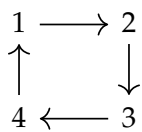
\[
\begin{tikzcd}
1 \arrow[r] & 2 \arrow[d] \\
4 \arrow[u] & \arrow[l] 3
\end{tikzcd}
\]
To incorporate both types of morphisms, we describe our dynamical systems using double categories and double functors and describe Runge-Kutta as a double transformation.
To capture how dynamical systems can be combined, we also impose a monoidal product to describe dynamical systems being put side-by-side without interacting with each other, similar to what is described in Spivak's work, and we prove that our double functors and double transformations preserve this monoidal structure.

The next section, Section \ref{definitions}, defines what double categories, double functors, and double transformations are, along with their monoidal versions. 
To understand these definitions, however, the reader needs to know the definitions of categories, functors, and natural transformations, in addition to the monoidal forms of each of these concepts.
Spivak's paper describes continuous systems using a functor $\CS$ and discrete dynamical systems using a functor $\DS$ from the category of wiring diagrams to the category of sets.
In Section \ref{dynamical}, we define the double category of wiring diagrams $\mathscr{W}$ and describe our continuous dynamical systems through a double functor $\CS: \mathscr{W} \to \Prof$ and our four-step dynamical systems through a double functor $\DS/C_4 : \mathscr{W} \to \Prof$, where $\Prof$ is the double category of profunctors.
We then describe a monoidal product to represent combining dynamical systems together and show that our double category and double functors preserve this product.
Finally, in Section \ref{runge-kutta}, we describe Runge-Kutta as a double transformation between $\CS$ and $\DS/C_4$ and prove the following theorem.
\begin{restatable}{thm}{rkthm}
The Runge-Kutta method is a monoidal double transformation between $\CS$ and $\DS/ C_4$.
\end{restatable}

\section{Definitions} \label{definitions}

We begin by defining double categories, double functors, and double transformations, using the works of Shulman \cite{shulman:1}\cite{shulman:2}.
As mentioned in the introduction, double categories allow us to define two types of morphisms for our objects.

\begin{definition}
A \textit{(pseudo) double category} $\mbb{D}$ consists of a ``category of objects $\mbb{D}_0$'' and a ``category of arrows'' $\mbb{D}_1$, with structure functors
\[
U: \mbb{D}_0 \to \mbb{D}_1
\]
\[
L,R: \mbb{D}_1 \to \mbb{D}_0
\]
\[
\odot: \mbb{D}_1 \times_{\mbb{D}_0} \mbb{D}_1 \to \mbb{D}_1
\]
where the pullback is over 

\begin{center}
\begin{tikzcd}
\mbb{D}_1 \arrow{r}{R} & \mbb{D}_0 & \arrow[l,swap,"L"] \mbb{D}_1
\end{tikzcd}
\end{center}
These structure functors must satisfy the relations

\begin{center}
\begin{tabular}{rl}
$L(U_A) = A$ &$R(U_A) = A$\\
$L(M \odot N) = L(M)$ &$R(M \odot N) = R(N)$
\end{tabular}
\end{center}
and are equipped with natural transformations
\[
a: (M \odot N) \odot P \to M \odot (N \odot P)
\]
\[
l: U_{L(M)} \odot M \to M
\]
\[
r: M \odot U_{R(M)} \to M
\]
such that $L(a)$, $R(a)$, $L(l)$, $R(l)$, $L(r)$, $R(r)$ are all identities and the following diagrams commute:

(Triangle)
\begin{center}
\begin{tikzcd}
(x \odot U_{Rx}) \odot y \arrow[rr,"a"] \arrow[dr,"r \odot 1"'] && x \odot (U_{Ly} \odot y) \\
& x \odot y \arrow[ur,"1 \odot l"'] &
\end{tikzcd}
\end{center}

(Pentagon)
\begin{center}
\begin{tikzcd}
 & (w \odot x) \odot (y \odot z) \arrow[dr,"a"]\\
((w \odot x) \odot y) \odot z \arrow[ur,"a"] \arrow[d,"a"']& & (w \odot (x \odot (y \odot z))) \\
(w \odot (x \odot y)) \odot z \arrow[rr,"a"'] && w \odot ((x \odot y) \odot z) \arrow[u,"a"']
\end{tikzcd}
\end{center}
\end{definition}

\begin{definition}
Let $\mbb{D}$ and $\mbb{E}$ be double categories. A \textit{lax double functor} $F: \mbb{D} \to \mbb{E}$ consists of the following:
\begin{itemize}
\item Functors $F_0: \mbb{D}_0 \to \mbb{E}_0$ and $F_1: \mbb{D}_1 \to \mbb{E}_1$ such that $L \circ F_1 = F_0 \circ L$ and $R \circ F_1 = F_0 \circ R$.

\item Natural Transformations $F_{\odot} : F_1 M \odot F_1 N \to F_1(M \odot N)$ and $F_U : U_{F_0 -} \to F_1(U_{-})$, whose components are globular , and for which the following diagrams commute.
\end{itemize}

(Associativity)
\begin{center}
\begin{tikzcd}
(F_1(x) \odot F_1(y)) \odot F_1(z) \arrow[d,"F_{\odot} \odot id"'] \arrow[r,"a"] & F_1(x) \odot (F_1(y) \odot F_1(z)) \arrow[d,"id \odot F_{\odot}"]\\
F_1(x \odot y) \odot F_1(z) \arrow[d,"F_\odot"'] & F_1(x) \odot F_1(y \odot z) \arrow[d,"F_\odot"]\\
F_1((x \odot y) \odot z) \arrow[r,"F_1(a)"'] & F_1(x \odot (y \odot z))
\end{tikzcd}
\end{center}

(Unitality)
\begin{center}
\begin{tikzcd}
U_{L \circ F_1x} \odot F_1(x) \arrow[d,"l"'] \arrow[r,"F_u \odot id"] & F_1(U_{Lx}) \odot F_1x \arrow[d,"F_\odot"]\\
F_1(x) & \arrow[l,"F_1(l)"]  F_1(U_{Lx} \circ x)
\end{tikzcd} \begin{tikzcd}
F_1(x) \odot U_{F_0 \circ Rx} \arrow[d,"r"'] \arrow[r,"id \odot F_U"] & F_1(x) \odot F_1(U_{Rx}) \arrow[d,"F_\odot"]\\
F_1(x) & \arrow[l,"F_1(r)"] F_1(x \odot U_{Rx})
\end{tikzcd}
\end{center}

If $F_U$ is an isomorphism, then $F$ is called a normal double functor.
\end{definition}

\begin{definition}
A \textit{double transformation} between two lax double functors $\alpha: F \Rightarrow G: \mbb{D} \to \mbb{E}$ consists of natural transformations $\alpha_0: F_0 \Rightarrow G_0: \mbb{D}_0 \to \mbb{E}_0$ and $\alpha_1 : F_1 \Rightarrow G_1: \mbb{D}_1 \to \mbb{E}_1$ such that $L(\alpha_M) = \alpha_{LM}$ and $R(\alpha_M) = \alpha_{RM}$ and

\[
\begin{tikzcd}
FA \arrow[drr,phantom,"F_{\odot}"] \arrow[d,equal] \arrow[r,tick,"FM"] & FB \arrow[r,tick,"FN"] & FC \arrow[d,equal]\\
FA\arrow[drr,phantom,"\alpha_{M \odot N}"] \arrow[d,"\alpha_A"'] \arrow[rr,tick,"F(M \odot N)"] && FC \arrow[d,"\alpha_C"]\\
GA \arrow[rr,tick,"G(M \odot N)"'] && GC
\end{tikzcd} = 
\begin{tikzcd}
FA \arrow[dr,phantom,"\alpha_M"] \arrow[d,"\alpha_A"'] \arrow[r,tick,"FM"] & FB \arrow[dr,phantom,"\alpha_N"] \arrow[d,"\alpha_D"]\arrow[r,tick,"FN"] & FC \arrow[d,"\alpha_C"] \\
GA \arrow[drr,phantom,"G_{\odot}"] \arrow[d,equal] \arrow[r,tick,"GM"] & GB \arrow[r,tick,"GN"] & GC \arrow[d,equal]\\
GA \arrow[rr,tick,"G(M \odot N)"'] && GC
\end{tikzcd}
\]

\[
\begin{tikzcd}
FA \arrow[dr,phantom," F_U"] \arrow[d,equal] \arrow[r,tick,"U_{FA}"] & FA \arrow[d,equal] \\
FA \arrow[dr,phantom,"\alpha_{U_A}"] \arrow[d,"\alpha_A"'] \arrow[r,tick,"F(U_A)"] & FA \arrow[d,"\alpha_A"]\\
GA \arrow[r,tick,"G(U_A)"'] & GA
\end{tikzcd} = 
\begin{tikzcd}
FA  \arrow[dr,phantom,"U_{\alpha_A}"] \arrow[d,"\alpha_A"'] \arrow[r,tick,"U_{FA}"] & FA \arrow[d,"\alpha_A"]\\
GA \arrow[dr,phantom, "G_U"] \arrow[d,equal] \arrow[r,tick,"U_{GA}"] & GA \arrow[d,equal]\\
GA \arrow[r,tick,"G(U_A)"'] & GA
\end{tikzcd}
\]
\end{definition}

We now describe the monoidal versions of these definitions. The definitions for the monoidal double functor and monoidal double transformation are the same as the definitions Shulman uses for monoidal framed bicategories.

\begin{definition}
A \textit{monoidal double category} is a double category equipped with functors $\otimes: \mathbb{D} \times \mathbb{D}$ and $I: * \to \mbb{D}$, and invertible transformations
\[
\otimes \circ (\Id \times \otimes) \cong \otimes \circ (\otimes \times \Id)
\]
\[
\otimes \circ (\Id \times I) \cong \Id
\]
\[
\otimes \circ (I \times \Id) \cong \Id
\]
satisfying the usual axioms.

Unpacking this definition more explicitly, we see that a monoidal double category is a double category together with the following structure.
\begin{itemize}
\item $\mbb{D}_0$ and $\mbb{D}_1$ are both monoidal categories.
\item If $I$ is the monoidal unit of $\mbb{D}_0$, then $U_I$ is the monoidal unit of $\mbb{D}_I$.
\item The functors $L$ and $R$ are strict monoidal, i.e. $L(M \otimes N) = LM \otimes LN$ and $R(M \otimes N) = RM \otimes RN$ and $L$ and $R$ also preserve the associativity and unit constraints.
\item We have globular isomorphisms
\[
\mathfrak{x}: (M_1 \otimes N_1) \odot (M_2 \otimes N_2) \xrightarrow{\cong} (M_1 \odot M_2) \otimes (N_1 \odot N_2)
\]
and
\[
\mathfrak{u}: U_{A \otimes B} \xrightarrow{\cong} (U_A \otimes U_B)
\]
such that the following diagrams commute:
\begin{center}
\begin{tikzcd}
((M_1 \otimes N_1) \odot (M_2 \otimes N_2)) \odot (M_3 \otimes N_3)\arrow[d] \arrow[r] & ((M_1 \odot M_2) \otimes (N_1 \odot N_2)) \odot (M_3 \otimes N_3) \arrow[d]\\
(M_1 \otimes N_1) \odot ((M_2 \otimes N_2) \odot (M_3 \otimes N_3)) \arrow[d] & ((M_1 \odot M_2) \odot M_3) \otimes ((N_1 \odot N_2) \odot N_3) \arrow[d] \\
(M_1 \otimes N_1) \odot ((M_2 \odot M_3) \otimes (N_2 \odot N_3)) \arrow[r] & (M_1 \odot (M_2 \odot M_3)) \otimes (N_1 \odot (N_2 \odot N_3))
\end{tikzcd}

\begin{tikzcd}
(M \times N) \odot U_{C \otimes D} \arrow[d]\arrow[r] & (M \otimes N) \circ (U_C \otimes U_D) \arrow[d]\\
M \otimes N & \arrow[l] (M \otimes U_C) \otimes (N \odot U_D)
\end{tikzcd}

\begin{tikzcd}
U_{A \otimes B} \odot (M \otimes N) \arrow[r] \arrow[d]& (U_A \otimes U_B) \odot (M \otimes N) \arrow[d] \\
M \otimes N & \arrow[l] (U_A \odot M) \otimes (U_B \odot N)
\end{tikzcd}
\end{center}
(these arise from the constraint data from the pseudo double functor $\otimes$).
\item The following diagrams commute, expressing that the associativity isomorphism for $\otimes$ is a transformation of double categories.

\begin{center}
\begin{tikzcd}
((M_1 \otimes N_1) \otimes P_1) \odot ((M_2 \otimes N_2) \otimes P_2) \arrow[d]\arrow[r] & (M_1 \otimes (N_1 \otimes P_1)) \odot (M_2 \otimes (N_2 \otimes P_2)) \arrow[d]\\
((M_1 \otimes N_1) \odot (M_2 \otimes N_20) \otimes (P_1 \odot P_2) \arrow[d]& (M_1 \odot M_2) \otimes ((N_1 \otimes P_1) \odot (N_2 \otimes P_2)) \arrow[d]\\
((M_1 \otimes M_2) \otimes (N_1 \odot N_2)) \otimes (P_1 \odot P_2) \arrow[r] & (M_1 \odot M_2) \otimes ((N_1 \odot N_2) \otimes (P_1 \odot P_2))
\end{tikzcd}

\begin{tikzcd}
U_{(A \otimes B) \otimes C} \arrow[r] \arrow[d] & U_{A \otimes (B \otimes C)} \arrow[d] \\
U_{A \otimes B} \otimes U_C \arrow[d] & U_A \otimes U_{B \otimes C} \arrow[d] \\
(U_A \otimes U_B) \otimes U_C \arrow[r] & U_A \otimes (U_B \otimes U_C)
\end{tikzcd}
\end{center}

\item The following diagrams commute, expressing that the unit isomorphisms for $\otimes$ are transformations of double categories.

\begin{center}
\begin{tikzcd}
(M \otimes U_I) \odot (N \otimes U_I) \arrow[r] \arrow[d] & (M \odot N) \otimes (U_I \odot U_I) \arrow[d]\\
M \odot N & \arrow[l] (M \odot N) \otimes U_I
\end{tikzcd} \begin{tikzcd}
U_{ A \otimes I} \arrow[r] \arrow[rd] & U_A \otimes U_I \arrow[d]\\
 & U_A
\end{tikzcd}

\begin{tikzcd}
(U_I \otimes M) \odot (U_I \otimes N) \arrow[r] \arrow[d] & (U_I \odot U_I) \otimes (M \odot N) \arrow[d] \\
M \odot N & \arrow[l] U_I \otimes (M \odot N)
\end{tikzcd} \begin{tikzcd}
U_{I \otimes A} \arrow[dr] \arrow[r] & U_I \otimes U_A \arrow[d] \\
 & U_A
\end{tikzcd}
\end{center}
\end{itemize}
\end{definition}

A good example of a monoidal double category is $\Prof$, \textit{the double category of categories and profunctors}. One can check that imposing the Cartesian product as the monoidal product gives us a monoidal double category.

\begin{definition}
A \textit{lax monoidal double functor} between monoidal double categories $\mbb{D},\mbb{E}$ consists of the following structure and properties.
\begin{itemize}
\item A lax double functor $F: \mbb{D} \to \mbb{E}$.
\item The structure of a lax monoidal functor on $F_0$ and $F_1$.
\item Equalities $LF_1 = F_0L$ and $RF_1 = F_0 R$ of lax monoidal functors.
\item The composition constraints for the lax double functor $F$ are monoidal natural transformations.
\end{itemize}
\end{definition}

\begin{definition}
A \textit{monoidal double transformation} is a double transformation such that $\alpha_0$ and $\alpha_1$ are monoidal natural transformations.
\end{definition}

\section{Dynamical Systems as a Double Functor} \label{dynamical}

We can now describe our wiring boxes and wiring diagrams as a double category, allowing us to describe one morphism using wiring diagrams and another morphism using commuting diagrams which describe a wiring box acting in step with another.

\begin{definition}
We define $\mathscr{W}$, \textit{the double category of wiring diagrams}, by the following categories $\mathscr{W}_0$ and $\mathscr{W}_1$. 
$\mathscr{W}_0$ is the category of wiring diagrams.
The objects are ordered pairs $(A,B)$ with $A,B \in \mathbf{Set}$ and the morphisms are wiring diagrams.

\begin{center}
\begin{tikzcd}
(A,B) \arrow[r,"\varphi"] & (C,D)
\end{tikzcd}
\end{center}

\[
\varphi^{in}: C \times B \to A
\]
\[
\varphi^{out}: B \to D
\]
The composition of wiring diagrams $\varphi$ and $\psi$
\begin{center}
\begin{tikzcd}
(A,B) \arrow[r,"\varphi"] & (C,D) \arrow[r,"\psi"] & (E,F)
\end{tikzcd}
\end{center}
is the wiring diagram with $(\psi \circ \varphi)^{in} : E \times B \to A$ given by
\begin{center}
\begin{tikzcd}
E \times B \arrow[r,"id \times \Delta"] &[5pt] E \times B \times B \arrow[r,"id \times \varphi^{out} \times id"] &[30pt] E \times D \times B \arrow[r,"\psi^{in} \times id"] &[15pt] C \times B \arrow[r,"\varphi^{in}"] & A
\end{tikzcd}
\end{center}
and $(\psi \circ \varphi)^{out} : B \to F$ given by
\begin{center}
\begin{tikzcd}
B \arrow[r,"\varphi^{out}"] & D \arrow[r,"\psi^{out}"] & F
\end{tikzcd}
\end{center}

We define $\mathscr{W}_1$ to be the following category.
The objects are ordered pairs $(f,g)$ where $f: A \to A'$ and $g: B \to B'$ with $L(f,g) = (A,B)$ and $R(f,g) = (A',B')$.
We define a morphism $(f,g)$ to $(f',g')$ to be a pair of wiring diagrams $(\varphi_1,\varphi_2)$ such that the following diagrams commute.

\begin{center}
\begin{tikzcd}
B_1 \arrow[d,"\varphi_1^{out}"'] \arrow[r,"g"] & B_2 \arrow[d,"\varphi_2^{out}"]\\
B_1' \arrow[r,"g'"'] & B_2'
\end{tikzcd} \begin{tikzcd}
A_1' \times B_1 \arrow[d,"\varphi_1^{in}"'] \arrow[r,"f' \times g"] & A_2' \times B_2 \arrow[d,"\varphi_2^{in}"]\\
A_1 \arrow[r,"f"'] & A_2
\end{tikzcd}
\end{center}
The resulting 2-cell can be described by the following box.
\begin{center}
\begin{tikzcd}
(A_1,B_1) \arrow[d,"\varphi_1"'] \arrow[r,tick,"{(f,g)}"] & (A_2,B_2) \arrow[d,"\varphi_2"] \\
(A_1',B_1') \arrow[r,tick,"{(f',g')}"'] & (A_2',B_2')
\end{tikzcd}
\end{center}

Then we let the composition of two morphisms in $\mathscr{W}_1$ be given by the following equality.
\[
\begin{tikzcd}
(A_1',B_1') \arrow[d,"\varphi'_1"'] \arrow[r,"{(f',g')}"] & (A_2',B_2') \arrow[d,"\varphi_2'"] \\
(A_2',B_2') \arrow[r,"{(f'',g'')}"'] & (A_2'',B_2'')
\end{tikzcd} \circ \begin{tikzcd}
 (A_1,B_1) \arrow[d,"\varphi_1"'] \arrow[r,"{(f,g)}"] & (A_2,B_2) \arrow[d,"\varphi_2"]\\
(A_1',B_1') \arrow[r,"{(f',g')}"'] & (A_2',B_2')
\end{tikzcd} = \begin{tikzcd}
(A_1,B_1) \arrow[d,"\varphi_1' \circ \varphi_1"'] \arrow[r,"{(f,g)}"] & (A_2,B_2) \arrow[d,"\varphi_2' \circ \varphi_2"] \\
(A_1'',B_1'') \arrow[r,"{(f'',g'')}"'] & (A_2'',B_2'')
\end{tikzcd}
\]
The latter term is a valid morphism since we have the following commuting diagrams.

\begin{center}
\begin{tabular}{cc}
\begin{tikzcd}
B_1 \arrow[d,"\varphi_1^{out}"'] \arrow[r,"g"] & B_2 \arrow[d,"\varphi_2^{out}"] \\
B_1' \arrow[d,"{\varphi''_1}^{out}"']\arrow[r,"g'"] & B_2' \arrow[d,"{\varphi'_2}^{out}"] \\
B_1'' \arrow[r,"g''"'] & B_2''
\end{tikzcd} & \begin{tikzcd}
A_1'' \times B_1 \arrow[d,"id \times \Delta"'] \arrow[r,"f'' \times g"] & A_2'' \times B_2 \arrow[d,"id \times \Delta"]\\
A_1'' \times B_1 \times B_1 \arrow[d,"id \times \varphi_1^{out} \times id"'] \arrow[r,"f'' \times g \times g"] & A_2'' \times B_2 \times B_2 \arrow[d,"id \times \varphi_2^{out} \times id"]\\
A''_1 \times B_1' \times B_1 \arrow[d,"{\varphi_1'}^{in} \times id"'] \arrow[r,"f'' \times g' \times g"] & A_2'' \times B_2' \times B_2 \arrow[d,"{\varphi_2'}^{in} \times id"] \\
A'_1 \times B_1 \arrow[d,"\varphi_1^{in}"] \arrow[r,"f' \times g"'] & A'_2 \times B_2 \arrow[d,"\varphi_2^{in}"] \\
A_1 \arrow[r,"f"'] & A_2
\end{tikzcd}
\end{tabular}
\end{center}

We define horizontal composition for objects by $(f',g') \odot (f,g) = (f' \circ f,g' \circ g)$ and define horizontal composition for morphisms by taking the unique morphism having the required left and right values.
One can check that the identity natural transformations work for $a,l,r$.
\end{definition}

Now we impose the following monoidal product to represent the combination of two wiring boxes.

\[
\left(\begin{tikzcd}
(A_1,B_1) \arrow[r,"{(f,g)}"] & (A_2,B_2)
\end{tikzcd}\right)
\otimes
\left(\begin{tikzcd}
(A_1',B_1') \arrow[r,"{(f',g')}"] & (A_2',B_2')
\end{tikzcd}\right)
\]
\[=
\begin{tikzcd}
(A_1 \times A_1',B_1 \times B_1') \arrow[r,"{(f \times f',g \times g')}"] &[25pt] (A_2 \times A_2',B_2 \times B_2')
\end{tikzcd}
\]

\[
\begin{tikzcd}
(A_1,B_1) \arrow[d,"\varphi_1"] \\
(C_1,D_1)
\end{tikzcd} \otimes
\begin{tikzcd}
(A_2,B_2) \arrow[d,"\varphi_2"] \\
(C_2,D_2)
\end{tikzcd} = 
\begin{tikzcd}
(A_1 \times A_2, B_1 \times B_2) \arrow[d,"\varphi_1 \times \varphi_2"]\\
(C_1 \times C_2,D_1 \times D_2)
\end{tikzcd}
\]
In the second equality, the last wiring diagram $\varphi_1 \times \varphi_2$ is the product of wiring diagrams $\varphi_1$ and $\varphi_2$ and is given by the following functions.
\[
(\varphi_1 \times \varphi_2)^{out}(s_1 \times s_2) = \varphi_1^{out}(s_1) \times \varphi_2^{out}(s_2)
\]
\[
(\varphi_1 \times \varphi_2)^{in}(a_1 \times a_2,s_1 \times s_2) = \varphi_1^{in}(a_1,s_1) \times \varphi_2^{in}(a_2,s_2)
\]
Since the monoidal product of horizontal moprhisms and vertical morphisms act coordinate-wise on the objects, the monoidal product of 2-cells make sense using the monoidal product of the boundaries.

We prove that $\mathscr{W}_0$ forms a monoidal category.
Our product forms a functor since
\[
(\varphi_2^{out} \circ \varphi_1^{out}) \times (\psi_2^{out} \circ \psi_1^{out}) = (\varphi_2^{out} \times \psi_2^{out}) \circ (\varphi_1^{out} \times \psi_1^{out})
\]
 and 
\begin{multline*}
\varphi_1(\varphi_2(A_3,\varphi_1^{out}(B_1)),B_1) \times \psi_1(\psi_2(C_2,\psi_1^{out}(D_1)),D_1) =\\
 (\varphi_1 \times \psi_1)(\varphi_2 \times \psi_2)(A_2 \times C_2, \varphi_1 \times \psi_1(B_1 \times D_1)),B_1 \times D_1)
\end{multline*}
The unit object is $(*,*)$ and the unit morphism is the trivial wiring diagram $(*,*) \to (*,*)$.

$\mathscr{W}_1$ also forms a monoidal category with the following product.
\[
\begin{tikzcd}
(A_1,B_1) \arrow[r,"{(f_1,g_1)}"] \arrow[d,"\varphi"'] & (C_1,D_1) \arrow[d,"\psi"] \\
(A_2,B_2) \arrow[r,"{(f_2,g_2)}"'] & (C_2,D_2)
\end{tikzcd} \otimes \begin{tikzcd}
(A_1',B_1') \arrow[r,"{(f_1',g_1')}"] \arrow[d,"\varphi'"'] & (C_1',D_1') \arrow[d,"\psi'"] \\
(A_2',B_2') \arrow[r,"{(f_2',g_2')}"'] & (C_2',D_2')
\end{tikzcd}\]\[ = \begin{tikzcd}
(A_1 \times A_1',B_1 \times B_1') \arrow[r,"{(f_1 \times f_1',g_1 \times g_1')}"] \arrow[d,"\varphi \times \varphi'"'] &[25pt] (C_1 \times C_1', D_1 \times D_1') \arrow[d,"\psi \times \psi'"] \\
(A_2 \times A_2', B_2 \times B_2') \arrow[r,"{(f_2 \times f_2',g_2 \times g_2')}"'] &[25pt] (C_2 \times C_2', D_2 \times D_2')
\end{tikzcd}
\]
Our product forms a functor since everything works coordinate-wise, and we can take the unit object to be \begin{tikzcd} (*,*) \arrow[r,"{(!,!)}"] & (*,*) \end{tikzcd} and the unit morphism to be the box with the unit objects at the top and bottom and unit wiring diagrams at the left and right.
These products on $\mathscr{W}_0$ and $\mathscr{W}_1$ affect $L$ and $R$ in the same way so $L$ and $R$ are strict monoidal functors.
Finally, we can take $\mathfrak{x}$ and $\mathfrak{u}$ to be identities and the remaining conditions can be easily checked. 

%
%

%

As mentioned in the introduction, for continuous dynamical systems, our state set $S$ is a real vector space.
We define a morphism of continuous dynamical systems to be a commuting diagram
\[
\begin{tikzcd}
A \times S \arrow[d,"id \times m"'] \arrow[r,"f^{in}"] & S \arrow[d,"m"] \arrow[r,"f^{out}"] & B \arrow[d,"id"]\\
A \times T \arrow[r,"g^{in}"'] & T \arrow [r,"g^{out}"'] & B
\end{tikzcd}
\]
such that $m$ is a linear map.
If $m$ is an isomorphism, then we can say that the dynamical systems are isomorphic, and from this, we can split dynamical systems with inputs $A$ and $B$ into isomorphism classes.

We define the category $\CS(A,B)$, with the objects being the isomorphism classes of dynamical systems
\begin{center}
\begin{tikzcd}
A \times S \arrow[r,"f^{in}"] & S \arrow[r,"f^{out}"] & B
\end{tikzcd}
\end{center}
where $S$ is a real vector space and given representatives of two isomorphism classes, say with state sets $S$ and $T$, we define a morphism to be a diagram
\begin{center}
\begin{tikzcd}
A \times S \arrow[d,"id \times m"'] \arrow[r,"f^{in}"] & S \arrow[d,"m"] \arrow[r,"f^{out}"] & B \arrow[d,"id"]\\
A \times T \arrow[r,"g^{in}"'] & T \arrow [r,"g^{out}"'] & B
\end{tikzcd}
\end{center}
with $m: S \to T$ a linear map such that the diagram commutes.

We can now define the map $\CS: \mathscr{W} \to \Prof$ to represent our continuous dynamical systems.
We define $\CS_0$ by sending $(A,B) \in \mathscr{W}_0$ to the category $\CS(A,B)$.
To show that this is a functor, we need to show that each wiring diagram $\varphi: (A,B) \to (A',B')$ gives a functor $CS(A,B) \to CS(A',B')$.
The wiring diagram $\varphi$ act on morphisms in $\CS(A,B)$ by the following commuting diagrams.

\begin{center}
\begin{tabular}{cc}
\begin{tikzcd}
A' \times S \arrow[d,"id \times m"'] \arrow[r,"id \times \Delta"] & A' \times S \times S \arrow[d,"id \times m \times m"] \arrow[r,"id \times out \times id"] &[15pt] A' \times B \times S \arrow[d,"id \times id \times m"] \arrow[r,"\varphi^{in} \times id"] & A\times S \arrow[d,"id \times m"] \arrow[r,"in"] & S \arrow[d,"m"]\\
A' \times S' \arrow[r,"id \times \Delta"'] & A' \times S' \times S' \arrow[r,"id \times out \times id"'] &[15pt] A' \times B \times S' \arrow[r,"\varphi^{in} \times id"'] & A \times S' \arrow[r, "in"'] & S'
\end{tikzcd} &
\begin{tikzcd}
S \arrow[d,"m"'] \arrow[r,""] & B \arrow[d,"id"] \arrow[r,"\varphi^{out}"] & B' \arrow[d,"id"]\\
S' \arrow[r,""] & B \arrow[r,"\varphi^{out}"'] & B'
\end{tikzcd}
\end{tabular}
\end{center}
Since these diagrams respect the composition of morphisms in $\CS(A,B)$, $\CS(\varphi)$ is a functor.
We can also see that $\CS(\varphi) \circ \CS(\psi) = \CS(\varphi \circ \psi)$ because both sides send change the inputs and outputs of a dynamical system in the same way.

%
%

We then define $\CS_1: \mathscr{W}_1 \to \Prof_1$ to be the map given by sending \begin{tikzcd}
(A,B) \arrow[r,"{(f,g)}"] & (C,D) \in \mathscr{W}_1
\end{tikzcd} to the profunctor $\CS(A,B)^{op} \times \CS(C,D) \to \Set$ where 
\[
(\begin{tikzcd}
A \times S \arrow[r] & S \arrow[r] & B
\end{tikzcd}, \begin{tikzcd}
C \times T \arrow[r] & T \arrow[r] & D
\end{tikzcd})
\]
 is mapped to the set of linear maps $m$ such that the following diagram commutes.
\begin{center}
\begin{tikzcd}
A \times S \arrow[d,"f \times m"'] \arrow[r] & S \arrow[d,"m"] \arrow[r] & B\arrow[d,"g"]\\
C \times T \arrow[r] & T \arrow[r] & D
\end{tikzcd}
\end{center}
We send the morphism
\[
 \begin{tikzcd}
(A,B) \arrow[d,"\varphi"'] \arrow[r,tick,"{(f,g)}"] & (C,D) \arrow[d,"\psi"]\\
(A',B') \arrow[r,tick,"{(f',g')}"'] & (C',D')
\end{tikzcd} \in \mathscr{W}_1
\]
to the natural transformation sending
\begin{center}
\begin{tikzcd}
A \times S \arrow[d,"f \times m"'] \arrow[r] & S \arrow[d,"m"] \arrow[r] & B\arrow[d,"g"]\\
C \times T \arrow[r] & T \arrow[r] & D
\end{tikzcd}
\end{center}
to
\begin{center}
\begin{tikzcd}
A' \times S \arrow[d,"f' \times m"'] \arrow[r] & S \arrow[d,"m"] \arrow[r] & B\arrow[d,"g'"]\\
C' \times T \arrow[r] & T \arrow[r] & D
\end{tikzcd}
\end{center}
by applying $\varphi$ and $\psi$ to the corresponding inputs and outputs.
This is a valid transformation because of how the squares in $\mathscr{W}_1$ commute.

%
%

Since the chosen profunctors for $\CS(U(x))$ align with the hom functor of $\CS(x)$ for any $x \in \mathscr{W}_0$, we can take $\CS_U$ to be the identity natural transformation.
We define $\CS_{\odot}$ by assigning to each element $(f,g) \times (f',g') \in \mathscr{W}_1 \times_{\mathscr{W}_0} \mathscr{W}_1$ the morphism in $\Prof_1$ sending 
\[
\begin{tikzcd}
A \times S \arrow[d] \arrow[r] & S \arrow[d,"f"] \arrow[r] & B \arrow[d]\\
A' \times S' \arrow[d] \arrow[r] & S' \arrow[d,"g"] \arrow[r] & B'\arrow[d]\\
A'' \times S'' \arrow[r] & S'' \arrow[r] & B''
\end{tikzcd}
\]
to 
\[
\begin{tikzcd}
A \times S \arrow[d] \arrow[r] & S \arrow[d, "g \circ f"]\arrow[r] & B \arrow[d]\\
A'' \times S'' \arrow[r] & S'' \arrow[r] & B'
\end{tikzcd}
\]
where each of the dynamical systems above are the representatives of their isomorphism classes.
This gives us a transformation because the equivalence classes generated by profunctor composition give the same result under our dynamical system morphism composition.
Associativity follows from the fact that function composition is associative and unitality follows from the fact that our profunctors and composition transformation work the same way as unit profunctors and their composition.

Now we show that $\CS$ is a lax monoidal double functor.
First, we show that $\CS_0$ and $\CS_1$ are lax monoidal functors.
We define $\epsilon$ to be the untial morphism of $\CS_0$ sending the trivial category to the trivial dynamical system $\begin{tikzcd} * \times * \arrow[r] & * \arrow[r] & *\end{tikzcd}$ in $\CS(*,*)$ and we define $\delta$ be the unital morphism of $\CS_1$ sending the trivial profunctor to the trivial hom from the trivial dynamical system to itself.
For monoidal natural transformations, we define $\mu$ to be the monoidal natural transformation of $\CS_0$ which sends
\[
(\begin{tikzcd} A_1 \times S_1 \arrow[r] & S_1 \arrow[r] & B_1 \end{tikzcd},\begin{tikzcd}A_2 \times S_2 \arrow[r] & S_2 \arrow[r] & B_2 \end{tikzcd})
\]
 to 
\[
\begin{tikzcd} (A_1 \times A_2) \times (S_1 \times S_2) \arrow[r] & S_1 \times S_2 \arrow[r] & B_1 \times B_2\end{tikzcd}.
\]
and define $\nu$ to be the monoidal natural transformation of $\CS_1$ which sends an ordered pair of dynamical system morphisms to the product of the morphisms.
Since everything works componentwise, the associativity and unitality constraints for our lax monoidal functors hold.
We already know that $L\CS_1 = \CS_0 L$ and $R\CS_1 = \CS_0 R$ as functors so we just need to check the lax monoidal properties.
All of these functors send ordered pairs of categories of dynamical systems to the category of product dynamical systems in the same way, so the equalities hold.
Finally, we need to check that the composition constraints are monoidal natural transformations.
The unital transformation is indeed monoidal since it is the identity transformation.
For the composition transformation, we need to show that the following diagram commutes.
\[
\begin{tikzcd}
(\CS(M_1) \odot \CS(N_1)) \otimes (\CS(M_2) \odot \CS(N_2)) \arrow[r] \arrow[d] & \CS(M_1 \odot N_1) \otimes \CS(M_2 \odot N_2) \\
\CS(M_1 \otimes M_2) \odot \CS(N_1 \otimes N_2) \arrow[r] & \CS((M_1 \otimes M_2) \odot (N_1 \otimes N_2))
\end{tikzcd}
\]
This can be seen from the fact that both paths send the pair of composed dynamical system morphisms
\[
\left(\begin{tikzcd}
A \times S \arrow[d] \arrow[r] & S \arrow[d] \arrow[r] & B \arrow[d]\\
A' \times S' \arrow[d] \arrow[r] & S' \arrow[d] \arrow[r] & B'\arrow[d]\\
A'' \times S'' \arrow[r] & S'' \arrow[r] & B''
\end{tikzcd},\begin{tikzcd}
C \times T \arrow[d] \arrow[r] & T \arrow[d] \arrow[r] & D \arrow[d]\\
C' \times T' \arrow[d] \arrow[r] & T' \arrow[d] \arrow[r] & D'\arrow[d]\\
C'' \times T'' \arrow[r] & T'' \arrow[r] & D''
\end{tikzcd}\right)
\]
to the morphism
\[
\begin{tikzcd}
(A \times C) \times (S \times T) \arrow[d] \arrow[r] & S \times T \arrow[d] \arrow[r] & B \times D \arrow[d]\\
(A'' \times C'') \times (S'' \times T'') \arrow[r] & S'' \times T'' \arrow[r] & B'' \times D''
\end{tikzcd}
\]
One path does this by taking the product morphism first, then composing, while the other path composes the morphisms and then takes the product.
Since both paths give the same result, we see that $\CS$ is a lax monoidal double functor.


We also define the category of four-step dynamical systems.
A four-step dynamical system is a discrete dynamical system along with a commuting diagram
\[
\begin{tikzcd}
A \times S \arrow[d] \arrow[r] & S \arrow[d] \arrow[r] & B \arrow[d]\\
* \times \{1,2,3,4\} \arrow[r] & \{1,2,3,4\} \arrow[r] & * \\
\end{tikzcd}
\]
where the bottom dynamical system behaves by cyclically iterating through all four elements.
We will denote this bottom dynamical system by $C_4$.
Another way we can portray this dynamical system is
\[
\begin{tikzcd}
A \times (S_1 \sqcup S_2 \sqcup S_3 \sqcup S_4) \arrow[r] & (S_1 \sqcup S_2 \sqcup S_3 \sqcup S_4) \arrow[r] & B
\end{tikzcd}
\]
where each $S_i$ is mapped to $i$.
A morphism between two four-step dynamical systems is a map $f: S \to T$ such that $S_i$ maps into $T_i$ and the diagram
\[
\begin{tikzcd}
A \times S \arrow[d,"f \times m"'] \arrow[r] & S \arrow[d,"m"] \arrow[r] & B\arrow[d,"g"]\\
C \times T \arrow[r] & T \arrow[r] & D
\end{tikzcd}
\]
commutes.
As in the case for $\CS$, we describe $(\DS/ C_4)(A,B)$ to be the category of isomorphism classes of four-step dynamical systems with input set $A$ and output set $B$.
Then we define $(\DS/ C_4)_0$ to be the category with objects of the form $(\DS / C_4)(A,B)$ and morphisms being wiring diagrams, and we define $(\DS/ C_4)_1$ to be the category of profunctors describing the morphisms of dynamical system classes with different inputs or outputs.
Like we did for $\CS$, we take $(\DS / C_4)_U$ to be the identity and $(\DS / C_4)_\odot$ to be the composition natural transformation.
The proof that these give a lax monoidal double functor is very similar to the proof for $\CS$.

%
%

\section{Runge-Kutta as a Double Transformation} \label{runge-kutta}

\begin{definition}
The Runge-Kutta map $\RK_h(A,B): \CS(A,B) \to \DS/ C_4(A,B)$ is defined by sending a dynamical system
\begin{center}
\begin{tikzcd}
A \times S \arrow[r] & S \arrow[r] & B
\end{tikzcd}
\end{center}
to the dynamical system
\begin{center}
\begin{tikzcd}
A \times (S \sqcup S^2 \sqcup S^3 \sqcup S^4) \arrow[d] \arrow[r] & S \sqcup S^2 \sqcup S^3 \sqcup S^4 \arrow[d] \arrow[r] & B\arrow[d]\\
* \times \{1,2,3,4\} \arrow[r] & \{1,2,3,4\} \arrow[r] & *
\end{tikzcd}
\end{center}
where 
\begin{align*}
{f'}^{upd}(a,s) &= (a,s,f^upd(s))\\
{f'}^upd(a,s,k_1) &= (a,s,k_1,f^upd(s + \frac{h}{2} k_1))\\
{f'}^upd(a,s,k_1,k_2) &= (a,s,k_1,k_2,f^upd(s + \frac{h}{2} k_2))\\
{f'}^upd(a,s,k_1,k_2,k_3) &= (a,s + \frac{h}{6} k_1 + \frac{h}{3} k_2 + \frac{h}{3} k_3 + \frac{h}{6} f^upd(a,s + hk_3))
\end{align*}
and
\begin{align*}
f'^rdt(a,s) &= s\\
f'^rdt(a,s,k_1) &= s + \frac{h}{2} k_1\\
f'^rdt(a,s,k_1,k_2) &= s + \frac{h}{2}k_2\\
f'^rdt(a,s,k_1,k_2,k_3) &= s + hk_3\\
\end{align*}
\end{definition}

\begin{lemma}
Let $X \in \CS(A,B)$ be a dynamical system. Then for any wiring diagram $\varphi: (A,B) \to (C,D)$, we have
\[
\RK_h(\varphi(X)) = \varphi(\RK_h(X))
\]
\end{lemma}

\begin{proof}
Let $X$ be the dynamical system
\[
\begin{tikzcd}
A \times S \arrow[r,"f^{upd}"] & S \arrow[r,"f^{out}"] & B
\end{tikzcd}
\]
Then $\varphi(X)$ is the dynamical system
\begin{tikzcd}
C \times S \arrow[r,"g^{upd}"] & S \arrow[r,"g^{out}"] & B
\end{tikzcd}
where $g^{upd}(c,s) = f^{upd}(\varphi^{in}(c,f^{out}(s)),s)$ and $g^{out} = \varphi^{out}(f^{out}(s))$.
We then compute that $\RK_h(\varphi(X))$ is the dynamical system
\[
\begin{tikzcd}
C \times S \sqcup S^2 \sqcup S^3 \sqcup S^4 \arrow[r,"h_1^{upd}"] & S \sqcup S^2 \sqcup S^3 \sqcup S^4 \arrow[r,"h_1^{out}"] & D
\end{tikzcd}
\]
with
\begin{align*}
h_1^{upd}(c,s) &= (c,s,g^{upd}(c,s)) = (c,s,f^{upd}(\varphi^{in}(c,f^{out}(s)),s))\\
h_1^{upd}(c,s,k_1) &= (c,s,k_1,g^{upd}(c,s + \frac{h}{2} k_1)) = (c,s,k_1,f^{upd}(\varphi^{in}(c,f^{out}(s + \frac{h}{2} k_1)),s + \frac{h}{2} k_1))\\
h_1^{upd}(c,s,k_1,k_2) &= (c,s,k_1,k_2,g^{upd}(c,s + \frac{h}{2} k_2)) = (c,s,k_1,k_2,f^{upd}(\varphi^{in}(c,f^{out}(s + \frac{h}{2} k_2)),s + \frac{h}{2} k_2))\\
h_1^{upd}(c,s,k_1,k_2,k_3) &= (c,s + \frac{h}{6} k_1 + \frac{h}{3} k_2 + \frac{h}{3} k_3 + \frac{h}{6} g^{upd}(c,s + hk_3)) \\
&= (c,s + \frac{h}{6} k_1 + \frac{h}{3} k_2 + \frac{h}{3} k_3 + \frac{h}{6} f^{upd}(\varphi^{in}(c,f^{out}(s + hk_3)),s + hk_3)
\end{align*}
\end{proof}

Finally, we restate Theorem 1 and prove it.

\rkthm*

\begin{proof}

We first show that $(RK_h)_0$ defines a natural transformation $\CS_0 \to (\DS/ C_4)_0$.
For this, we need to show that $RK_h$ gives a functor from $\CS_0(A,B)$ to $(\DS/ C_4)(A,B)$.
This fact follows from the above lemma and the fact that we are sending the $\CS$-morphism
\[
\begin{tikzcd}
A \times S \arrow[d,"\varphi \times m"'] \arrow[r,"f^{upd}"] & S \arrow[d,"m"]\arrow[r,"f^{out}"] & B \arrow[d,"\varphi"]\\
A \times S' \arrow[r,"f'^{upd}"'] & S' \arrow[r,"f'^{out}"'] & B\\
\end{tikzcd}
\]
to the $\DS/ C_4$-morphism
\[
\begin{tikzcd}
A \times (S \sqcup S^2 \sqcup S^3 \sqcup S^4) \arrow[d,"{\varphi \times (m \sqcup m^2 \sqcup m^3 \sqcup m^4)}"'] \arrow[r] & S \sqcup S^2 \sqcup S^3 \sqcup S^4 \arrow[d,"m \sqcup m^2 \sqcup m^3 \sqcup m^4"] \arrow[r] & B\arrow[d,"\varphi"]\\
A \times (S' \sqcup S'^2 \sqcup S'^3 \sqcup S'^4) \arrow[r] & S' \sqcup S'^2 \sqcup S'^3 \sqcup S'^4 \arrow[r] & B
\end{tikzcd}
\]
Since the morphisms are component-wise, we have that the functoral properties hold.
Now we just need to show that these morphisms in $\Prof$ behave naturally with $\CS$ and $\DS / C_4$, but this follows from our above Lemma, so we can see that $(RK_h)_0$ gives us a natural transformation between $\CS_0$ and $(\DS / C_4)_0$.

We then show that $(RK_h)_1$ gives a natural transformation $\CS_1$ to $(\DS / C_4)_1$.
This is not hard to check since we are sending 
\[
\begin{tikzcd}
A \times S \arrow[d,"\varphi \times m"'] \arrow[r,"f^{upd}"] & S \arrow[d,"m"]\arrow[r,"f^{out}"] & B \arrow[d,"\varphi"]\\
A' \times S' \arrow[r,"f'^{upd}"] & S' \arrow[r,"f'^{out}"] & B'\\
\end{tikzcd}
\]
to
\[
\begin{tikzcd}
A \times (S \sqcup S^2 \sqcup S^3 \sqcup S^4) \arrow[d,"{\varphi \times (m \sqcup m^2 \sqcup m^3 \sqcup m^4)}"'] \arrow[r] & S \sqcup S^2 \sqcup S^3 \sqcup S^4 \arrow[d,"m \sqcup m^2 \sqcup m^3 \sqcup m^4"] \arrow[r] & B\arrow[d,"\varphi"]\\
A' \times (S' \sqcup S'^2 \sqcup S'^3 \sqcup S'^4) \arrow[r] & S' \sqcup S'^2 \sqcup S'^3 \sqcup S'^4 \arrow[r] & B'
\end{tikzcd}
\]
and all the morphisms are component-wise.
One can check the diagrams in the definition commute since our functions act component-wise and also compose component-wise.

Finally, we just need to show that these natural transformations are monoidal so that we have a lax monoidal double transformation.
For this, we need the natural transformation to commute with our monoidal transformations.
This is true for both natural transformations, since multiplying the dynamical systems and then applying Runge-Kutta gives an isomorphic dynamical system as applying Runge-Kutta to two dynamical systems and then merging the two systems.
The only difference is that the order of the products in the discrete sum may be different, but since our maps work component-wise, we still get isomorphic results, so our natural transformation indeed commutes with our monoidal transformations.
Our last requirement, the unital requirement for the natural transformations, is easy to check, so Runge-Kutta indeed gives a monoidal double transformation.
\end{proof}

\section{Acknowledgements}

The author would like to thank David Spivak for many helpful discussions.

\printbibliography

\end{document}